\documentclass[12pt]{amsart}
\usepackage{amscd,amssymb,amsmath,multicol}
\newtheorem{thm}[equation]{Theorem}
\numberwithin{equation}{section}
\newtheorem{cor}[equation]{Corollary}

\newtheorem{prop}[equation]{Proposition}
\newtheorem{tab}[equation]{Table}

\begin{document}
\raggedbottom \voffset=-.7truein \hoffset=0truein \vsize=8truein
\hsize=6truein \textheight=8truein \textwidth=6truein
\baselineskip=18truept
\def\mapright#1{\ \smash{\mathop{\longrightarrow}\limits^{#1}}\ }
\def\mapleft#1{\smash{\mathop{\longleftarrow}\limits^{#1}}}
\def\mapup#1{\Big\uparrow\rlap{$\vcenter {\hbox {$#1$}}$}}
\def\mapdown#1{\Big\downarrow\rlap{$\vcenter {\hbox {$\ssize{#1}$}}$}}
\def\mapne#1{\nearrow\rlap{$\vcenter {\hbox {$#1$}}$}}
\def\mapse#1{\searrow\rlap{$\vcenter {\hbox {$\ssize{#1}$}}$}}
\def\mapr#1{\smash{\mathop{\rightarrow}\limits^{#1}}}
\def\ss{\smallskip}
\def\ssum{\sum\limits}
\def\dsum{{\displaystyle{\sum}}}
\def\vp{v_1^{-1}\pi}
\def\at{{\widetilde\alpha}}
\def\sm{\wedge}
\def\la{\langle}
\def\ra{\rangle}
\def\on{\operatorname}
\def\spin{\on{Spin}}
\def\kbar{{\overline k}}
\def\qed{\quad\rule{8pt}{8pt}\bigskip}
\def\ssize{\scriptstyle}
\def\a{\alpha}
\def\bz{{\Bbb Z}}
\def\im{\on{im}}
\def\ct{\widetilde{C}}
\def\ext{\on{Ext}}
\def\sq{\on{Sq}}
\def\eps{\epsilon}
\def\ar#1{\stackrel {#1}{\rightarrow}}
\def\br{{\bold R}}
\def\bc{{\bold C}}
\def\bN{{\bold N}}
\def\si{\sigma}
\def\Ebar{{\overline E}}

\def\tfrac{\textstyle\frac}
\def\tb{\textstyle\binom}
\def\Si{\Sigma}
\def\w{\wedge}
\def\equ{\begin{equation}}
\def\b{\beta}
\def\G{\Gamma}
\def\g{\gamma}
\def\psit{\widetilde{\Psi}}
\def\tht{\widetilde{\Theta}}

\def\fbar{{\overline f}}
\def\endeq{\end{equation}}
\def\sn{S^{2n+1}}
\def\zp{\bold Z_p}
\def\A{{\cal A}}
\def\P{{\cal P}}
\def\cj{{\cal J}}
\def\zt{{\bold Z}_2}
\def\bs{{\bold s}}
\def\bof{{\bold f}}
\def\bq{{\bold Q}}
\def\be{{\bold e}}
\def\Hom{\on{Hom}}
\def\ker{\on{ker}}
\def\coker{\on{coker}}
\def\da{\downarrow}
\def\colim{\operatornamewithlimits{colim}}
\def\zphat{\bz_2^\wedge}
\def\io{\iota}
\def\Om{\Omega}
\def\u{{\cal U}}
\def\e{{\cal E}}
\def\exp{\on{exp}}
\def\line{\rule{.6in}{.6pt}}
\def\wbar{{\overline w}}
\def\xbar{{\overline x}}
\def\ybar{{\overline y}}
\def\zbar{{\overline z}}
\def\ebar{{\overline \be}}
\def\nbar{{\overline n}}
\def\rbar{{\overline r}}
\def\Ubar{{\overline U}}
\def\et{{\widetilde e}}
\def\ni{\noindent}
\def\coef{\on{coef}}
\def\den{\on{den}}
\def\lcm{\on{l.c.m.}}
\def\vi{v_1^{-1}}
\def\ot{\otimes}
\def\psibar{{\overline\psi}}
\def\mhat{{\hat m}}
\def\exc{\on{exc}}
\def\ms{\medskip}
\def\ehat{{\hat e}}
\def\etao{{\eta_{\text{od}}}}
\def\etae{{\eta_{\text{ev}}}}
\def\dirlim{\operatornamewithlimits{dirlim}}
\def\gt{\widetilde{L}}
\def\lt{\widetilde{\lambda}}
\def\sgd{\on{sgd}}
\def\ord{\on{ord}}
\def\gd{{\on{gd}}}
\def\rk{{{\on{rk}}_2}}
\def\nbar{{\overline{n}}}

\def\N{{\Bbb N}}
\def\Z{{\Bbb Z}}
\def\Q{{\Bbb Q}}
\def\R{{\Bbb R}}
\def\C{{\Bbb C}}
\def\l{\left}
\def\r{\right}
\def\ls{\leq}
\def\gs{\geq}
\def\bg{\bigg}
\def\({\bigg(}
\def\[{\bigg[}
\def\){\bigg)}
\def\]{\bigg]}
\def\colon{{:}\;}
\def\lg{\langle}
\def\rg{\rangle}
\def\t{\text}
\def\f{\frac}
\def\mo{\on{mod}}
\def\vexp{v_1^{-1}\exp}
\def\al{\alpha}
\def\ve{\varepsilon}
\def\bi{\binom}
\def\eq{\equiv}
\def\cs{\cdots}
\def\dstyle{\displaystyle}
\def\Remark{\noindent{\it  Remark}}
\title
{$v_1$-periodic 2-exponents of $SU(2^e)$ and $SU(2^e+1)$}
\author{Donald M. Davis}
\address{Department of Mathematics, Lehigh University\\Bethlehem, PA 18015, USA}
\email{dmd1@lehigh.edu}
\date{September 22, 2011}

\keywords{homotopy groups, special unitary groups, exponents, $v_1$-periodicity}
\thanks {2000 {\it Mathematics Subject Classification}:
55Q52,11B73.}

\maketitle
\begin{abstract} We determine precisely the largest $v_1$-periodic homotopy groups
of $SU(2^e)$ and $SU(2^e+1)$. This gives new results about the largest actual homotopy
groups of these spaces. Our proof relies on results about 2-divisibility of restricted
sums of binomial coefficients times powers proved by the author in a companion paper.
\end{abstract}
\section{Main result}\label{intro}
The 2-primary $v_1$-periodic homotopy groups, $\vi\pi_i(X)$, of a topological space $X$ are a
localization of a first approximation to its 2-primary homotopy groups. They are roughly the portion of $\pi_*(X)$ detected by 2-local $K$-theory.(\cite{Bo})
If $X$ is a sphere or compact Lie group, each $v_1$-periodic homotopy
group of $X$ is a direct summand of some actual homotopy group of $X$.(\cite{DM})

Let
$$T_j(k)=\sum_{\text{odd }i}\tbinom ji i^k$$
denote one family of partial Stirling numbers. In \cite{partial}, the author obtained several results about
$\nu(T_j(k))$, where $\nu(n)$ denotes the exponent of 2 in $n$. Some of those will be used in this paper, and will
be restated as needed.

Let
$$\be(k,n)=\min(\nu(T_j(k)):\ j\ge n).$$
It was proved in \cite[1.1]{BDSU} (see also \cite[1.4]{DP}) that $\vi\pi_{2k}(SU(n))$
is isomorphic to $\bz/2^{\be(k,n)-\eps}$ direct sum with possibly one or two $\bz/2$'s.
Here $\eps=0$ or 1, and $\eps=0$ if $n$ is odd or if  $k\equiv n-1$ mod 4, which are the
only cases required here.

Let $$s(n)=n-1+\nu([n/2]!).$$ It was proved in \cite{DS1} that $\be(n-1,n)\ge s(n)$. Let
$$\ebar(n)=\max(\be(k,n):k\in\Z).$$
Thus $\ebar(n)$ is what we might call the $v_1$-periodic 2-exponent of $SU(n)$. Then clearly
\begin{equation}\label{comp3}s(n)\le\be(n-1,n)\le\ebar(n),\end{equation}
and calculations suggest that both of these inequalities are usually quite close to being equalities.
In \cite[p.22]{D23}, a table is given comparing the numbers in (\ref{comp3}) for $n\le38$.

Our main theorem verifies a conjecture of \cite{D23} regarding the values in (\ref{comp3})
when $n=2^e$ or $2^e+1$.
\begin{thm}\label{main} \qquad

\begin{enumerate}
\item[a.]
If $e\ge3$, then $\be(k,2^e)\le 2^e+2^{e-1}-1$ with equality occurring iff $k\equiv 2^e-1\mod 2^{2^{e-1}+e-1}$.
\item[b.] If $e\ge2$, then $\be(k,2^e+1)\le 2^e+2^{e-1}$ with equality occurring iff $k\equiv 2^e+2^{2^{e-1}+e-1}\mod 2^{2^{e-1}+e}$.
\end{enumerate}
\end{thm}
Thus the values in (\ref{comp3}) for $n=2^e$ and $2^e+1$ are as in Table \ref{tab1}.

\bigskip
\begin{minipage}{6.5in}
\begin{tab}\label{tab1}{Comparison of values}
\begin{center}
\begin{tabular}{c|ccc}
$n$&$s(n)$&$\be(n-1,n)$&$\ebar(n)$\\
\hline
$2^e$&$2^e+2^{e-1}-2$&$2^e+2^{e-1}-1$&$2^e+2^{e-1}-1$\\
$2^e+1$&$2^e+2^{e-1}-1$&$2^e+2^{e-1}-1$&$2^e+2^{e-1}$
\end{tabular}
\end{center}
\end{tab}
\end{minipage}
\bigskip

Note that $\ebar(n)$ exceeds $s(n)$ by 1 in both cases, but for different reasons.
When $n=2^e$, the largest value occurs for $k=n-1$, but is 1 larger than the general bound
established in \cite{DS1}. When $n=2^e+1$, the general bound for $\be(n-1,n)$ is sharp,
but a larger group occurs when $n-1$ is altered in a specific way.
The numbers $\ebar(n)$ are interesting, as they give what are probably the largest
2-exponents in $\pi_*(SU(n))$, and this is the first time that infinite families of
these numbers have been computed precisely.

The homotopy $2$-exponent of a topological space $X$, denoted $\exp_2(X)$, is the largest $k$ such that $\pi_*(X)$ contains an element of order $2^k$.
An immediate corollary of Theorem \ref{main} is
\begin{cor} For $\eps\in\{0,1\}$ and $2^e+\eps\ge5$,
$$\exp_2(SU(2^e+\eps))\ge 2^e+2^{e-1}-1+\eps.$$
\end{cor}
\noindent This result is 1 stronger than the result noted in \cite[Thm 1.1]{DS1}.

Theorem \ref{main} is implied by the following two results. The first will be proved in Section \ref{pfsec}. The second is \cite[Thm 1.1]{partial}.
\begin{thm}\label{five} Let $e\ge3$.
\begin{enumerate}
\item[i.] If $\nu(k)\ge e-1$, then $\nu(T_{2^e}(k))=2^e-1$.
\item[ii.] If $j\ge 2^e$ and $\nu(k-(2^e-1))\ge 2^{e-1}+e-1$, then $\nu(T_j(k))\ge 2^e+2^{e-1}-1$.
\item[iii.] If $j\ge 2^e+1$ and $\nu(k-2^e)=2^{e-1}+e-1$, then $\nu(T_j(k))\ge 2^e+2^{e-1}$.
\end{enumerate}
\end{thm}

\begin{thm}\label{thm0} $($\cite[1.1]{partial}$)$ Let $e\ge2$, $n=2^e+1$ or $2^e+2$, and $1\le i\le 2^{e-1}$.
 There is a 2-adic integer $x_{i,n}$ such that for all integers $x$
$$\nu(T_n(2^{e-1}x+2^{e-1}+i))=\nu(x-x_{i,n})+n-2.$$
Moreover
$$\nu(x_{i,{2^e+1}})\begin{cases}=i&\text{if $i=2^{e-2}$ or $2^{e-1}$}\\
>i&\text{otherwise.}\end{cases}$$
and
$$\nu(x_{i,{2^e+2}})\begin{cases}=i-1&\text{if }1\le i\le 2^{e-2}\\
=i&\text{if }2^{e-2}<i<2^{e-1}\\
>i&\text{if }i=2^{e-1}.\end{cases}$$
\end{thm}

Regarding small values of $e$: \cite[\S8]{DP} and \cite[Table 1.3]{partial} make it clear that the results stated in this section
for $T_n(-)$, $\be(-,n)$ and $SU(n)$ are valid for small values of $n\ge5$ but not for $n<5$.

\begin{proof}[Proof that Theorems \ref{five} and \ref{thm0} imply Theorem \ref{main}] For part (a):
Let $k\equiv 2^e-1$ mod $2^{2^{e-1}+e-1}$. Theorem \ref{five}(ii) implies $\be(k,2^e)\ge 2^e+2^{e-1}-1$, and  \ref{thm0} with $n=2^e+2$, $i=2^{e-1}-1$,
and $\nu(x)\ge 2^{e-1}$ implies that equality is obtained for such $k$.

To see that $\be(k,2^e)<2^e+2^{e-1}-1$ if $k\not\equiv2^e-1$ mod $2^{2^{e-1}+e-1}$, we write $k=i+2^{e-1}x+2^{e-1}$ with $1\le i\le2^{e-1}$.
We must show that for each $k$ there exists some $j\ge 2^e$ for which $\nu(T_j(k))<2^e+2^{e-1}-1$.
\begin{itemize}
\item If $i=2^{e-1}$, we use \ref{five}(i).
\item If $i=2^{e-2}$, we use \ref{thm0} with $n=2^e+1$ if $\nu(x)<2^{e-2}$ and with $n=2^e+2$ if $\nu(x)\ge 2^{e-2}$.
\item For other values of $i$, we use \ref{thm0} with $n=2^e+1$ if $\nu(x)\le i$ and with $n=2^e+2$ if $\nu(x)>i$,
except in the excluded case $i=2^{e-1}-1$ and $\nu(x)>i$.
\end{itemize}

For part (b):
Let $k\equiv 2^e+2^{2^{e-1}+e-1}$ mod $2^{2^{e-1}+e}$. Theorem \ref{five}(iii) implies $\be(k,2^e+1)\ge 2^e+2^{e-1}$, and  \ref{thm0} with $n=2^e+2$,  $i=2^{e-1}$,
and $\nu(x)= 2^{e-1}$ implies that equality is obtained for such $k$.

To see that $\be(k,2^e)<2^e+2^{e-1}$ if $k\not\equiv2^e+2^{2^{e-1}+e-1}$ mod $2^{2^{e-1}+e}$, we write $k=i+2^{e-1}x+2^{e-1}$ with $1\le i\le2^{e-1}$.
\begin{itemize}
\item If $i=2^{e-1}$, we use \ref{thm0} with $n=2^e+1$ unless $\nu(x)=2^{e-1}$, which case is excluded.
\item If $i=2^{e-2}$, we use \ref{thm0} with $n=2^e+2$ if $\nu(x)=2^{e-2}$ and with $n=2^e+1$ otherwise.
\item If $1\le i<2^{e-2}$, we use \ref{thm0} with $n=2^e+1$ if $\nu(x)=i-1$ and with $n=2^e+2$ otherwise.
\item If $2^{e-2}< i<2^{e-1}$, we use \ref{thm0} with $n=2^e+1$ if $\nu(x)=i$ and with $n=2^e+2$ otherwise.
\end{itemize}

\end{proof}

The proof  does not make it transparent why the largest value of $\be(k,n)$ occurs when $k=n-1$ if $n=2^e$ but not if
$n=2^e+1$.  The following example  may shed some light.
We consider the illustrative case $e=4$. We wish to see why
\begin{itemize}
\item $\be(k,16)\le23$ with equality iff $k\equiv 15\mod 2^{11}$, while
\item $\be(k,17)\le 24$ with equality iff $k\equiv 16+2^{11}\mod2^{12}$.
\end{itemize}

Tables \ref{tbl15} and \ref{tbl16} give relevant values of $\nu(T_j(k))$.

\medskip
\begin{minipage}{6.5in}
\begin{tab}\label{tbl15}{Values of $\nu(T_j(k))$ relevant to $\be(k,16)$}
\begin{center}
\begin{tabular}{cr|cccc}
&&&$j$&&\\
&&$16$&$17$&$18$&$19$\\
\hline
&$7$&$24$&{\bf 19}&$20$&$20$\\
$\nu(k-15)$&$8$&$25$&{\bf 20}&$21$&$21$\\
&$9$&$26$&{\bf 21}&$22$&$22$\\
&$10$&$27$&{\bf 22}&$\ge24$&$\ge 24$\\
&$11$&$\ge29$&$\ge24$&{\bf 23}&$23$\\
&$\ge12$&$28$&{\bf 23}&$23$&$23$
\end{tabular}
\end{center}
\end{tab}
\end{minipage}
\medskip

\medskip

\medskip
\begin{minipage}{6.5in}
\begin{tab}\label{tbl16}{Values of $\nu(T_j(k))$ relevant to $\be(k,17)$}
\begin{center}
\begin{tabular}{cr|cccc}
&&&$j$&&\\
&&$17$&$18$&$19$&$20$\\
\hline
&$8$&{\bf 20}&$21$&$22$&$23$\\
$\nu(k-16)$&$9$&{\bf 21}&$22$&$23$&$24$\\
&$10$&{\bf 22}&$23$&$\ge25$&$\ge26$\\
&$11$&$\ge24$&{\bf 24}&$24$&$25$\\
&$12$&{\bf 23}&$\ge26$&$24$&$25$\\
&$\ge13$&{\bf 23}&$25$&$24$&$25$
\end{tabular}
\end{center}
\end{tab}
\end{minipage}
\medskip

The values $\be(k,16)$ and $\be(k,17)$ are the smallest entry in a row, and are listed in boldface.
The tables only include values of $k$ for which $\nu(k-(n-1))$ is rather large, as these give the largest
values of $\nu(T_j(k))$. Larger values of $j$ than those tabulated will give larger values of $\nu(T_j(n))$.
Note how each column has the same general form, leveling off after a jump. This reflects the $\nu(x-x_{i,n})$
in Theorem \ref{thm0}. The prevalence of this behavior is the central theme of \cite{partial}. The phenomenon which
we wish to illuminate here is how the bold values increase steadily until they level off in Table \ref{tbl15},
while in Table \ref{tbl16} they jump to a larger value before leveling off. This is a consequence of the synchronicity
of where the jumps occur in columns 17 and 18 of the two tables.

\section{Proof of Theorem \ref{five}}\label{pfsec}
In this section, we prove Theorem \ref{five}. The proof uses the following results from \cite{partial}.

\begin{prop}\label{DSthm} $($\cite[3.4]{DS1} or \cite[2.1]{partial}$)$ For any nonnegative integers $n$ and $k$,
$$\nu\bigl(\sum_i\tbinom n{2i+1}i^k\bigr)\ge\nu([n/2]!).$$
\end{prop}

The next result is a refinement of Proposition \ref{DSthm}. Here and throughout, $S(n,k)$ denote Stirling numbers of the second kind.
\begin{prop}\label{stirowr} $($\cite[2.3]{partial}$)$ Mod $4$
$${\tfrac 1{n!}}\sum_i\tbinom{2n+\eps}{2i+b}i^k\equiv\begin{cases}S(k,n)+2nS(k,n-1)&\eps=0,\,b=0\\
(2n+1) S(k,n)+2(n+1)S(k,n-1)&\eps=1,\,b=0\\
2nS(k,n-1)&\eps=0,\,b=1\\
S(k,n)+2(n+1)S(k,n-1)&\eps=1,\,b=1.
\end{cases}$$
\end{prop}

\begin{prop}\label{lcor1} $($\cite[2.7]{partial}$)$ For $n\ge3$, $j>0$,
and $p\in\Z$, $$\nu(\sum\tbinom{n}{2i+1}(2i+1)^pi^j)\ge \max(\nu([\tfrac n2]!),n-\a(n)-j)$$ with equality if $n\in\{2^e+1,2^e+2\}$ and $j=2^{e-1}$. \end{prop}

Other well-known facts that we will use are
\begin{equation}\label{Sa} (-1)^jj!S(k,j)=\sum\tbinom j{2i}(2i)^k-T_j(k)\end{equation}
and
\begin{equation}\label{Smod2}S(k+i,k)\equiv \tbinom{k+2i-1}{k-1}\mod 2.\end{equation}
We also use that $\nu(n!)=n-\a(n)$, where $\a(n)$ denotes the binary digital sum of $n$, and that $\binom mn$ is odd iff, for all $i$,
$m_i\ge n_i$, where these denote the $i$th digit in the binary expansions of $m$ and $n$.

\begin{proof}[Proof of Theorem \ref{five}(i).] Using (\ref{Sa}), we have
$$T_{2^e}(2^{e-1}t)\equiv -S(2^{e-1}t,2^e)(2^e)!\mod 2^{2^{e-1}t},$$
and we may assume $t\ge2$ using the periodicity  of $\nu(T_n(-))$ established in \cite[3.12]{CK}.
But $S(2^{e-1}t,2^e)\equiv\binom{2^et-2^{e+1}+2^e-1}{2^e-1}\equiv1$ mod $2$. Since $\nu(2^e!)=2^e-1<2^{e-1}t$,
we are done. \end{proof}

\begin{proof}[Proof of parts (ii) and (iii) of Theorem \ref{five}.]
 These parts follow from (a) and (b) below by letting $p=2^e+\eps-1$ in (b), and adding.
\begin{enumerate}
\item[(a)] Let $\eps\in\{0,1\}$ and $n\ge2^e+\eps$.
$$\nu(T_n(2^e+\eps-1))\begin{cases}=2^e+2^{e-1}-1&\text{if }\eps=1\text{ and }n=2^e+1\\
\ge 2^e+2^{e-1}+\eps-1&\text{otherwise.}\end{cases}$$
\item[(b)] Let $p\in\Z$, $n\ge2^e$, and $\nu(m)\ge2^{e-1}+e-1$. Then
$$\nu\biggl(\sum\tbinom n{2i+1}(2i+1)^p\bigl((2i+1)^m-1\bigr)\biggr)\begin{cases}=2^e+2^{e-1}-1&\text{if }n=2^e+1\text{ and}\\
&\nu(m)=2^{e-1}+e-1\\
\ge2^e+2^{e-1}&\text{otherwise.}\end{cases}$$
\end{enumerate}

First we prove (a). Using (\ref{Sa}) and the fact that $S(k,j)=0$ if $k<j$, it suffices to prove
$$\nu\bigl(\sum\tbinom n{2i}i^{2^e+\eps-1}\bigr)\begin{cases}=2^{e-1}-1&\text{if }\eps=1\text{ and }n=2^e+1\\
\ge2^{e-1}&\text{otherwise,}\end{cases}$$
  and this is implied by Proposition \ref{DSthm}  if $n\ge 2^e+4$.
For $\eps=0$ and $2^e\le n\le 2^e+3$, by Proposition \ref{stirowr}
$$\nu(\sum\tbinom n{2i}i^{2^e-1})\ge 2^{e-1}-1+\min(1,\nu(S(2^e-1,2^{e-1}+\delta)))$$
with $\delta\in\{0,1\}$. The Stirling number here is easily seen to be even by (\ref{Smod2}).

Similarly $\nu(\sum\binom{2^e+1}{2i}i^{2^e})=2^{e-1}-1$ since $S(2^e,2^{e-1})$ is odd, and if $n-2^e\in\{2,3\}$, then
$\nu\bigl(\sum\tbinom n{2i}i^{2^e}\bigr)\ge2^{e-1}$ since $S(2^e,2^{e-1}+1)$ is even.

Now we prove part (b). The sum equals $\sum_{j>0} T_j$, where
$$T_j=2^j\tbinom mj\sum_i\tbinom n{2i+1}(2i+1)^pi^j.$$
We show that $\nu(T_j)=2^e+2^{e-1}-1$ if $n=2^e+1$, $j=2^{e-1}$, and $\nu(m)=2^{e-1}+e-1$, while in all other
cases, $\nu(T_j)\ge 2^e+2^{e-1}$.  If $j\ge2^e+2^{e-1}$, we use the $2^j$-factor. Otherwise, $\nu(\binom mj)=\nu(m)-\nu(j)$, and we use the first part of the max in Proposition \ref{lcor1} if $\nu(j)\ge e-1$, and the second part of the max otherwise.
\end{proof}

\def\line{\rule{.6in}{.6pt}}

\end{document}